\documentclass[11pt]{amsart}
\usepackage{amsfonts, amssymb, amsmath, amsthm, eufrak, color, float,enumerate}

\theoremstyle{plain}
   \newtheorem{teo}{Theorem}
   
   \newtheorem{lema}[teo]{Lemma}

\theoremstyle{definition}
   
\theoremstyle{remark}
 \newtheorem{obs}{Remark}

\numberwithin{equation}{section}
\allowdisplaybreaks

\newcommand{\R}{\mathbb{R}} 
\newcommand{\N}{\mathbb{N}} 
\newcommand{\norm}[1]{\|#1\|} 
\newcommand{\supp}{\mathrm{supp}}

\begin{document}

\title[Mixed weak estimates for fractional integral operators]{Mixed weak estimates of Sawyer type for fractional integrals and some related operators}

\author[F. Berra]{Fabio Berra}
\address{CONICET and Departamento de Matem\'{a}tica (FIQ-UNL),  Santa Fe, Argentina.}
\email{fberra@santafe-conicet.gov.ar}

\author[M. Carena]{Marilina Carena}
\address{CONICET (FIQ-UNL) and Departamento de Matem\'{a}tica (FHUC-UNL),  Santa Fe, Argentina.}
\email{marilcarena@gmail.com}

\author[G. Pradolini]{Gladis pradolini}
\address{CONICET and Departamento de Matem\'{a}tica (FIQ-UNL),  Santa Fe, Argentina.}
 \email{gladis.pradolini@gmail.com}

\thanks{The authors were supported by CONICET, UNL and ANPCyT}




\begin{abstract}
We prove mixed weak estimates of Sawyer type for fractional operators. More precisely, let $\mathcal{T}$ be either the maximal fractional function $M_\gamma$ or the fractional integral operator $I_\gamma$, $0<\gamma<n$, $1\leq p<n/\gamma$ and $1/q=1/p-\gamma/n$. If $u,v^{q/p}\in A_1$ or if $uv^{-q/{p'}}\in A_1$ and $v^q\in A_\infty(uv^{-q/{p'}})$ then we obtain that the estimate
\begin{equation*}
uv^{q/p}\left(\left\{x\in \R^n: \frac{|\mathcal{T}(fv)(x)|}{v(x)}>t\right\}\right)^{1/q}\leq \frac{C}{t}\left(\int_{\R^n}|f(x)|^pu(x)^{p/q}v(x)\,dx\right)^{1/p},
\end{equation*}
holds for every positive $t$ and every bounded function with compact support.

As an important application of the results above we further more exhibe mixed weak estimates for commutators of Calder\'on-Zygmund singular integral and fractional integral operators when the symbol $b$ is in the class Lipschitz-$\delta$, $0<\delta\leq 1$.
\end{abstract}

\maketitle

\thispagestyle{empty}

\section*{Introduction}

The first basic and deep result of the Muckenhoupt theory is the equivalence of the $A_p$ condition  of a weight $w$ with the boundedness
of the Hardy-Littlewood maximal operator $M$ on the associated weighted $L^p(w)$ space.
In other words, $w\in A_p$ if and only if
\begin{equation}\label{teo_de_muckenhoupt}
 \int_{\R^n}Mf(x)^pw(x)\,dx\leq C\int_{\R^n}|f(x)|^pw(x)\,dx.
\end{equation}
In \cite{Sawyer}, Sawyer noted that the fact that $w\in A_p$ implies $\eqref{teo_de_muckenhoupt}$ can be obtained from the factorization theorem of P. Jones (see \cite{jones}) together with a mixed inequality of an auxiliary operator. Precisely, given $w\in A_p$,  use the factorization theorem to write $w = uv^{1-p}$
with $u$ and $v$ in $A_1$, and consider the operator $S$ defined by
\[S(f)(x)=\frac{M(fv)(x)}{v(x)}.\]
It is not difficult  to see that the inequality
\begin{equation}\label{acotacion_de_S}
\int_{\R^n}|S(f)(x)|^pu(x)v(x)\,dx\leq C\int_{\R^n}|f(x)|^pu(x)v(x)\,dx
\end{equation}
is equivalent to \eqref{teo_de_muckenhoupt}.
%
That is, the strong $(p,p)$ type of $S$ with respect to the measure $\mu$ defined by $d\mu(x)=u(x)v(x)\,dx$ is equivalent to the weighted strong $(p,p)$ type of $M$. Since $S$  is bounded on $L^\infty(uv)$, then  \eqref{teo_de_muckenhoupt} will follow from interpolation techniques provided that the weak $(1,1)$ type of $S$, with respect to the measure $uv$, is achieved. Concretely, it is enough to show that the inequality
%
%
\[uv\left(\left\{x\in \R^n: S(f)(x)>t\right\}\right)\leq\frac{C}{t}\int_{\R^n}|f(x)|u(x)v(x)\,dx,\]
or equivalently
\begin{equation}\label{eq: mixed}
uv\left(\left\{x\in \R^n: \frac{M(fv)(x)}{v(x)}>t\right\}\right)\leq\frac{C}{t}\int_{\R^n}|f(x)|u(x)v(x)\,dx.
\end{equation}
holds for every $t>0$

The inequality \eqref{eq: mixed} was proved by
Sawyer in \cite{Sawyer} in $\mathbb R$  with
$u,v\in A_1$. The author also conjectured that the inequality  holds if $T$ is
the Hilbert transform. This conjecture was later proved in  \cite{CU-M-P} in a more general setting of Calder\'on-Zygmund operators as well as the Hardy-Littlewood maximal operator in $\R^n$. The conditions on the weights considered in \cite{CU-M-P} are not only $u,v \in A_1$ but also $u\in A_1$ and $v\in A_\infty(u)$.
For the latter class of weights, in  \cite{bcp} a similar result is proved for singular integral operators  $T$  whose kernels satisfy a $L^{\varphi}$-H\"{o}rmander property  for a given Young function $\varphi$. In the same article, mixed weak type inequalities for higher order commutators of $T$ with $BMO$ symbols were also proved.

The motivation of Sawyer impulses us to look for an adequate mixed inequality for other operators, such as the fractional maximal or the fractional integral operators. By following the ideas described above for the Hardy-Littlewood maximal operator $M$, we look for weighted mixed inequalities that allow us to deduce the boundedness of the fractional maximal operator $M_\gamma$, which was proved by Muckenhoupt and Wheeden in \cite{Muckenhoup-Wheeden-Int-Fraccionaria}. 
Concretely, they proved that if $0<\gamma<n$, $1<p<n/\gamma$, $1/q=1/p-\gamma/n$ and $w\in A_{p,q}$ (see section below) then the inequality

\begin{equation}\label{eq muck and wheeden int frac}
\left(\int_{\R^n}M_\gamma(f)(x)^qw(x)^q\,dx\right)^{1/q}\leq C\left(\int_{\R^n}|f(x)|^pw(x)^p\,dx\right)^{1/p}
\end{equation}
holds.



%

\medskip

Let $0<\gamma<n$, $1<p<n/\gamma$ and $1/q=1/p-\gamma/n$. Given  $w\in A_{p,q}$ we factorize  $w^q=uv^{1-r}$ with $u,v\in A_1$ and $r=1+q/{p'}$, and  we consider the operator
\[S(f)(x)=\frac{M_\gamma(fv)(x)}{v(x)}.\]
Let $z$ be the weight defined by $z(x)=u(x)^{1/q}v(x)^{1/p}$.
 We shall see that if $S$ is bounded from $L^p(z^p)$ to $L^q(z^q)$, then \eqref{eq muck and wheeden int frac} holds. Indeed, assume that
 \begin{equation}\label{acotacion_S_fraccionaria}
\left(\int_{\R^n}S(f)(x)^qz(x)^q\,dx\right)^{1/q}\leq C\left(\int_{\R^n}|f(x)|^pz(x)^p\,dx\right)^{1/p}.
\end{equation}
Then
\begin{align*}
\left(\int_{\R^n}M_\gamma(f)(x)^qw(x)^q\,dx\right)^{1/q}&=\left(\int_{\R^n}\frac{M_{\gamma}(fv^{-1}v)(x)^q}{v^{ q}}\,uv^{1-r}v^{ q}\,dx\right)^{1/q}\\
&=\left(\int_{\R^n}S(fv^{-1})^qz^q\,dx\right)^{1/q}\\
&\leq C\left(\int_{\R^n}|f(x)|^pv^{- p}z^p\,dx\right)^{1/p}\\
&=C\left(\int_{\R^n}|f(x)|^pw(x)^p\,dx\right)^{1/p}.
\end{align*}

Actually, \eqref{acotacion_S_fraccionaria} is equivalent to \eqref{eq muck and wheeden int frac}. Indeed, if \eqref{eq muck and wheeden int frac} holds then
\begin{align*}
\left(\int_{\R^n}S(f)(x)^qz(x)^q\,dx\right)^{1/q}&=\left(\int_{\R^n}M_\gamma(fv)(x)^qz(x)^qv(x)^{-q}\,dx\right)^{1/q}\\
&=\left(\int_{\R^n}M_\gamma(fv)(x)^qw(x)^q\,dx\right)^{1/q}\\
&\leq C\left(\int_{\R^n}f(x)^pv(x)^pw(x)^p\,dx\right)^{1/p}\\
&=C\left(\int_{\R^n}f(x)^pz(x)^p\,dx\right)^{1/p}.
\end{align*}

As in \cite{Muckenhoup-Wheeden-Int-Fraccionaria}, inequality \eqref{acotacion_S_fraccionaria} can be obtained if we prove that
\begin{equation*}
z^{q}\left(\left\{x\in \R^n: \frac{M_\gamma(fv)(x)}{v(x)}>t\right\}\right)^{1/q}\leq \frac{C}{t}\left(\int_{\R^n}|f(x)|^pz^{p}(x)\,dx\right)^{1/p},
\end{equation*}
for every $0<\gamma<n$, $1\leq p<n/\gamma$ and  $q$ satisfying $1/q=1/p-\gamma/n$, or equivalently
\begin{equation}\label{eq: eq main}
uv^{q/p}\left(\left\{x\in \R^n: \frac{M_\gamma(fv)(x)}{v(x)}>t\right\}\right)^{1/q}\leq \frac{C}{t}\left(\int_{\R^n}|f(x)|^pu^{p/q}(x)v(x)\,dx\right)^{1/p}.
\end{equation}

\medskip

We are now in a position to give our main results.
\begin{teo}\label{teo_main_Mgamma}
Let $0<\gamma<n$, $1\leq p<n/\gamma$ and $q$ satisfying $1/q=1/p-\gamma/n$. If $u,v$ are weights such that $u,v^{q/p}\in A_1$ or $uv^{-q/{p'}}\in A_1$ and $v^q\in A_\infty(uv^{-q/{p'}})$, then there exists a positive constant $C$ such that for every $t>0$
\[uv^{q/p}\left(\left\{x\in \R^n: \frac{M_\gamma(fv)(x)}{v(x)}>t\right\}\right)^{1/q}\leq \frac{C}{t}\left(\int_{\R^n}|f(x)|^pu(x)^{p/q}v(x)\,dx\right)^{1/p}.\]
\end{teo}

\begin{obs}
Let us point out that if $\gamma=0$ we get $p=q$, and in the particular case $p=q=1$ this result has already proved in \cite{CU-M-P}. Also, notice that the conditions on the weights coincide with the ones in that theorem for this case. For this reason, the theorem above can be seen as a generalization of the conditions of Theorem~\ref{mixed_para_M} not only for the case $p=1$ but also for other values of $p$.

\end{obs}

The next theorem establishes a similar estimate for the fractional integral operator $I_\gamma$.

\begin{teo}\label{teo_main_Tgamma}
Let $0<\gamma<n$, $1\leq p<n/\gamma$ and $q$ satisfying $1/q=1/p-\gamma/n$. If $u,v$ are weights such that $u,v^{q/p}\in A_1$ or $uv^{-q/{p'}}\in A_1$ and $v^q\in A_\infty(uv^{-q/{p'}})$, then there exists a positive constant $C$ such that for every $t>0$
\[uv^{q/p}\left(\left\{x\in \R^n: \frac{I_\gamma(fv)(x)}{v(x)}>t\right\}\right)^{1/q}\leq \frac{C}{t}\left(\int_{\R^n}|f(x)|^pu(x)^{p/q}v(x)\,dx\right)^{1/p}.\]
\end{teo}

\begin{obs}
If we set $\gamma=0$ and consider the case $p=q=1$, then this theorem was proved in \cite{CU-M-P} with $I_\gamma$ replaced with a Calder\'on-Zygmund operator $T$.
\end{obs}

We devote Section~\ref{section:main-fractional maximal} to prove Theorem~\ref{teo_main_Mgamma}. In Section~\ref{section:main-fractional integral} we consider the fractional integral operator and prove Theorem~\ref{teo_main_Tgamma}. Finally, in Section~\ref{section: applications} we apply the obtained results in order to prove mixed weak inequalities of Sawyer type for commutators of Calder\'on-Zygmund and fractional integral operators, both with Lipschitz symbol.

\section{Preliminaries, definitions and auxiliary lemmas}\label{seccion_preliminares_y_def}
Given $0<\gamma<n$, the \emph{fractional maximal operator} $M_\gamma$ is defined for $f\in L^{1}_\textrm{loc}$ by
\[M_\gamma f(x)=\sup_{Q\ni x} \frac{1}{|Q|^{1-\gamma/n}}\int_Q|f(y)|\,dy,\]
where the supremum is taken over all cubes containing $x$. The \emph{fractional integral operator} $I_\gamma$ is defined by
\[I_\gamma f(x)=\int_{\R^n}\frac{f(y)}{|x-y|^{n-\gamma}}\,dy.\]
It is well known that the boundedness of $I_\gamma$ can be obtained from the boundedness of $M_\gamma$, due to the following Coifman type inequality proved in \cite{Muckenhoup-Wheeden-Int-Fraccionaria}.

\begin{teo}[\cite{Muckenhoup-Wheeden-Int-Fraccionaria}]\label{teo_muckenhoupt_Tgamma_Mgamma}
Let $0< \gamma<n$ and $0<q<\infty$. For every weight $w\in A_\infty$ we have that
\[\int_{\R^n}|I_\gamma(f)(x)|^qw(x)\,dx\leq C\int_{\R^n}M_{\gamma}(f)(x)^qw(x)\,dx.\]
\end{teo}

In the proof of our main results, we shall use the following pointwise estimation for $M_\gamma$ in terms of the classical Hardy-Littewood maximal operator $M$. A proof can be found in \cite{G-P-S} for $1<p<n/\gamma$ in a more general context. For the sake of completeness we give the proof, including the case $p=1$.

\begin{lema}\label{desigualdad_puntual_Mfrac_M}
Let $0<\gamma<n$, $1\leq p<n/\gamma$, $1/q=1/p-\gamma/n$ and $s=1+q/{p'}$. Then,
\[M_\gamma(fw^{-1})(x)\leq M(f^{p/s}w^{-q/s})^{s/q}(x)\left(\int_{\R^n}f(y)^p\,dy\right)^{\gamma/n},\]
for every non-negative functions $w$  and $f\in L^p$.
\end{lema}

\begin{proof}
Set $g=f^{p/s}w^{-q/s}$, so that $fw^{-1}=g^{s/p}w^{q/p-1}$. Let us fix $x\in \R^n$ and let $Q$ be a cube containing $x$. Then, by applying  H\"{o}lder's inequality  with $n/(n-\gamma)$ and $n/\gamma$ we get
\begin{align*}
\frac{1}{|Q|^{1-\gamma/n}}\int_Q f(y)w^{-1}(y)\,dy&=\frac{1}{|Q|^{1-\gamma/n}}\int_Q g^{s/p}w^{q/p-1}\,dy\\
&=\frac{1}{|Q|^{1-\gamma/n}}\int_Q g^{1-\gamma/n}g^{s/p+\gamma/n-1}w^{q\gamma/n}\,dy\\
&\leq \left(\frac{1}{|Q|}\int_Q g(y)\,dy\right)^{1-\gamma/n}\left(\int_Q g^{(s/p+\gamma/n-1)n/\gamma}w^q\right)^{\gamma/n}\\
&\leq M(g)(x)^{s/q}\left(\int_{\R^n}f(y)^p\,dy\right)^{\gamma/n},
\end{align*}
since $1-\gamma/n=s/q$ and  $(s/p+\gamma/n-1)n/\gamma=s$, because the definition of $s$.
\end{proof}

\begin{obs}\label{rem puntual}
We want to point out that if we
apply Jensen's inequality with exponent $s>1$ in the average of $g$ in the previous proof, we get
\begin{equation}\label{eq: punctual}
M_{\gamma}(fw^{-1})(x)\leq M(f^pw^{-q})^{1/q}\left(\int_{\R^n}f^p(y)\,dy\right)^{\gamma/n}.
\end{equation}
We shall use this pointwise estimate later.
\end{obs}

Let us recall that a \emph{weight} $w$  is a locally integrable  function defined on
 $\R^n$, such that $0<w(x)<\infty$ a.e. $x\in \R^n$.
 For $1<p<\infty$ the \emph{Muckenhoupt $A_p$ class} is defined as the set of all weights $w$
 for which there exists a positive constant $C$ such that the inequality
\[\left(\frac{1}{|Q|}\int_Q w\right)\left(\frac{1}{|Q|}\int_Q w^{-\frac{1}{p-1}}\right)^{p-1}
\leq C\] holds for every cube $Q\subset \R^n$,  with sides parallel
to the coordinate axes.  For $p=1$, we say that $w\in A_1$ if there
exists a positive constant $C$ such that
 \begin{equation*}
  \frac{1}{|Q|}\int_Q w\le C\, \inf_Q w(x),
 \end{equation*}
 for every cube $Q\subset \R^n$. 
The smallest constant $C$ for which the Muckenhoupt condition holds
is called the $A_p$-constant of $w$, and denoted by $[w]_{A_p}$.
The $A_\infty$ class is defined by the collection of all the $A_p$ classes.
It is easy to see
that if $p<q$ then $A_p\subseteq A_q$. Given $1<p<\infty$, we use
$p'$ to denote the conjugate exponent $p/(p-1)$.
For $p=1$ we take $p'=\infty$. Some classical references for the basic theory of Muckenhoupt weights
 are for example~\cite{javi} and~\cite{garcia-rubio}.\\
 
For $1<p,q<\infty$ we say $w\in A_{p,q}$ if there exists a positive constant $C$ such that for every cube $Q\subset\R^n$
\[\left(\frac{1}{|Q|}\int_Qw^q\,dx\right)^{1/q}\left(\frac{1}{|Q|}\int_Qw^{-p'}\,dx\right)^{1/{p'}}\leq C.\]
When $p=1$ or $q=\infty$ the condition is written as follows
\[\left(\frac{1}{|Q|}\int_Qw^q\,dx\right)^{1/q}\|\mathcal{X}_Qw^{-1}\|_{\infty}\leq C,\]
\[\|\mathcal{X}_Qw\|_{\infty}\left(\frac{1}{|Q|}\int_Qw^{-p'}\,dx\right)^{1/{p'}}\leq C,\]
respectively. The smallest constant for which the conditions above hold is denoted by $[w]_{A_{p,q}}$.

These classes can be rewritten in terms of the $A_p$ classes. More precisely,
for $1<p,q<\infty$ we have that
\begin{enumerate}[(a)]
\item $w\in A_{p,q}$ if and only if $w^q\in A_{1+q/{p'}}$, where $p'=p/(p-1)$ denotes the conjugate exponent of $p$; (see \cite{Muckenhoup-Wheeden-Int-Fraccionaria})
\item $w\in A_{p,\infty}$ if and only if $w^{-p'}\in A_1$;
\item $w\in A_{1,q}$ if and only if $w^q\in A_1$.
\end{enumerate}

An important property of  Muckenhoupt weights is the \emph{reverse H\"{o}lder's condition}, which establishes that if
$w\in A_p$, for some $1\leq p\leq\infty$, then there exists a positive constant $C$
and $s>1$, such that for every cube $Q$
\begin{equation*}
\left(\frac{1}{|Q|}\int_Q w^s(x)\,dx\right)^{1/s}\leq
\frac{C}{|Q|}\int_Q w(x)\,dx.
\end{equation*}
We write $w\in
\textrm{RH}_s$ to point out that the inequality above holds,
and we denote by $[w]_{\textrm{RH}_s}$ the smallest constant $C$ for
which this condition holds and it only depends on $n$, $p$ and $[w]_{A_p}$.

A weight $w$ belongs to RH$_\infty$ if
there exists a positive constant $C$ such
that\begin{equation*}
\sup_Q w\leq\frac{C}{|Q|}\int_Q w,
\end{equation*}
for every $Q\subset \R^n$. Let us observe that
$\textrm{RH}_\infty\subseteq \textrm{RH}_s\subseteq \textrm{RH}_q$,
for every $1<q<s$. It can be also proved that $\textrm{RH}_s\subseteq A_\infty$.\\ 

The following lemmas will be very useful in the proof of our main results. The proof of the statements in the first one can be found in \cite{CU-N}.

\begin{lema}\label{lema_equivalencias}
The following statements hold.
\begin{enumerate}
\item $w\in A_{\infty}$ if and only if $w=w_1w_2$, with $w_1\in A_1$ and $w_2\in \textrm{RH}_{\infty}$.
\item If $w\in A_1$  then $w^{-1}\in\textrm{RH}_{\infty}$.
\item If $u,v\in\textrm{RH}_{\infty}$ then $uv\in\textrm{RH}_{\infty}$.
\end{enumerate}
\end{lema}

\begin{lema}\label{lema_producto_Ainf_RHinf}
Let $u$ and $v$ be two weights such that $u\in\textrm{RH}_{\infty}$ and $v\in A_\infty$. Then $uv\in A_\infty$.
\end{lema}

\begin{proof}
From Lemma~\ref{lema_equivalencias}, $v\in A_{\infty}$ implies that there exist $v_0\in A_1$ and  $v_1\in \textrm{RH}_{\infty}$ such that $v=v_0v_1$.
It will be enough to prove that $uv\in \textrm{RH}_{s}$, for some $s>1$. Since  $v_0\in A_1$ there exists  $s>1$ such that $v_0\in \textrm{RH}_{s}$. On the other hand, from Lemma~\ref{lema_equivalencias} we also have that $uv_1\in \textrm{RH}_{\infty}$. Then, given a cube $Q$ of $\R^n$ we have
\begin{align*}
\left(\frac{1}{|Q|}\int_Q (uv)^s\,dx\right)^{1/s}&=\left(\frac{1}{|Q|}\int_Q (uv_1v_0)^s\,dx\right)^{1/s}\\
&\leq \sup_Q{uv_1}\left(\frac{1}{|Q|}\int_Qv_0^s\,dx\right)^{1/s}\\
&\leq [uv_1]_{\textrm{RH}_\infty}[v_0]_{\textrm{RH}_s}\frac{1}{|Q|}\int_Q uv_1\,dx\frac{1}{|Q|}\int_Q v_0\,dx\\
&\leq[uv_1]_{\textrm{RH}_\infty}[v_0]_{\textrm{RH}_s}[v_0]_{A_1}\frac{1}{|Q|}\int_Q uv_0v_1\,dx\\
&=\frac{C}{|Q|}\int_Quv\,dx,
\end{align*}
as desired.
\end{proof}

Given weights $u$ and $v$, by $v \in  A_p(u)$ we mean that $v$ satisfies the $A_p$ condition
with respect to the measure $\mu$ defined as $d\mu=u(x)\, dx$. More precisely, for $1<p< \infty$, we say that
$v\in A_p(u)$ if there exists a positive constant $C$ such that
\[\left(\frac{1}{u(Q)}\int_Q v(x)u(x)\,dx \right)\left(\frac{1}{u(Q)}\int_Q v(x)^{-\frac{1}{p-1}}u(x)\,dx\right)^{p-1}\leq C,\]
for every cube $Q\subset\R^n$. As usual, $u(Q)=\int_Q u(x)\,dx$. A weight $v$ belongs to $A_1(u)$ if
\[\frac{1}{u(Q)}\int_Q v(x)u(x)\,dx\le C\, \inf_Q v(x).\]
We denote the collection of all the $A_p(u)$ classes by $A_\infty(u)$.\\
%



The following technical lemma will be useful in the proof of mixed estimates for $M_\gamma$.

\begin{lema}\label{lema_Mgamma_Mgamma0}
Let  $0<\gamma<n$, $1<p<n/\gamma$ and  $\gamma_0=p\gamma$. If $v\in A_1$ then for every bounded function $f$ with compact support we have
\[\frac{M_{\gamma}(fv)(x)}{v(x)}\leq [v]_{A_1}^{1/{p'}}\left(\frac{M_{\gamma_0}(f^pv)(x)}{v(x)}\right)^{1/p}.\]
\end{lema}

\begin{proof}
Fix $x\in\R^n$ and a cube $Q$ containing $x$. Then
\begin{align*}
\frac{1}{v(x)}\frac{1}{|Q|^{1-\gamma/n}}\int_Q fv^{1/p}v^{1/{p'}}&\leq \frac{1}{v(x)}\frac{1}{|Q|^{1-\gamma/n}}\left(\int_Q f^pv\right)^{1/p}\left(\int_Q v\right)^{1/{p'}}\\
&=\frac{1}{v(x)}\left(\frac{1}{|Q|^{1-\gamma_0/n}}\int_Q f^pv\right)^{1/p}\left(\frac{1}{|Q|}\int_Q v\right)^{1/p'}\\
&\leq \frac{[v]_{A_1}^{1/{p'}}}{v(x)}\left(M_{\gamma_0}(f^pv)(x)\right)^{1/p}v^{1/{p'}}(x)\\
&=[v]_{A_1}^{1/{p'}}\left(\frac{M_{\gamma_0}(f^pv)(x)}{v(x)}\right)^{1/p}.
\end{align*}
\end{proof}

\medskip
\medskip

\section{Mixed inequalities for $M_\gamma$}\label{section:main-fractional maximal}

The following result, which has been proved in \cite{CU-M-P} will let us find mixed inequalities for the operator $M_\gamma$, where $0<\gamma<n$.

\begin{teo}\cite[Thm. 1.3]{CU-M-P}\label{mixed_para_M}
If $u,v$ are weights such that $u,v\in A_1$, or $u\in A_1$ and $v\in
A_{\infty}(u)$, then there exists a positive constant $C$ such that for every
$t>0$,
\[uv\left(\left\{x\in\R^n: \frac{M(fv)(x)}{v(x)}>t\right\}\right)\leq \frac{C}{t}\int_{\R^n}|f(x)|u(x)v(x)\,dx,\]
where $M$ is the Hardy-Littlewood  maximal operator.
\end{teo}

\begin{proof}[Proof of Theorem~\ref{teo_main_Mgamma}]




Without loss of generality, since $M_\gamma{(fv)}=M_\gamma{(|f|v)}$ we can assume $f$ nonnegative. Let us consider first the case $p=1$. In this case, we have that $q=n/(n-\gamma)$ and $p'=\infty$, so that the hyphoteses on the weights are $u,v^q\in A_1$ or, $u\in A_1$ and $v^q\in A_\infty(u)$.
Set $w=u^{1/q}$ and $f_0=fwv$. We can also assume $f_0\in L^1$, since in other case the result trivially holds.
By applying Lemma~\ref{desigualdad_puntual_Mfrac_M} and Remark~\ref{rem puntual} to $f_0$ and $w$ we get
\[M_\gamma(fv)(x)=M_\gamma(f_0w^{-1})(x)\leq M(f_0w^{-q})^{1/q}(x)\left(\int_{\R^n}f_0(y)\,dy\right)^{\gamma/n}.\]
Then
\begin{align*}
uv^q\left(\left\{x\in \R^n: \frac{M_\gamma(fv)(x)}{v(x)}>t\right\}\right)^{1/q}&\leq uv^q\left(\left\{x\in \R^n: \frac{M(f_0w^{-q})(x)}{v^q(x)}>\left(\frac{t}{\left(\int_{\R^n}f_0\right)^{\gamma/n}}\right)^q\right\}\right)^{1/q}.
\end{align*}
From Theorem~\ref{mixed_para_M}, the last term can be bounded, for both conditions on the weights, by
\[C\,\frac{\left(\int_{\R^n}f_0\right)^{\gamma/n}}{t}\left(\int_{\R^n}f_0w^{-q}v^{-q}uv^q\,dy\right)^{1/q}=\frac{C}{t}\int_{\R^n}f(y)u^{1/q}(y)v(y)\,dy,\]
since $f_0w^{-q}u=f_0$ and $\gamma/n+1/q=1$.




Let us consider now the case $1<p<n/\gamma$. We shall begin with the assumption $u,v^{q/p}\in A_1$ and, as above,  we can assume that $\int_{\R^n}f^pu^{p/q}v\,dx$ is finite.
We have to prove that
\[uv^{q_0}\left(\left\{x\in \R^n: \frac{M_\gamma(fv)(x)}{v(x)}>t\right\}\right)^{1/{q_0}}\leq \frac{C}{t^p}\left(\int_{\R^n}|f(x)|^pu^{1/{q_0}}(x)v(x)\,dx\right),\]
where $q_0=q/p>1$. In order to see that inequality above holds, denote
$\gamma_0=p\gamma$. Then $1/{q_0}=p/q=1-p\gamma/n=1-\gamma_0/n$, and  $0<\gamma_0<n$ since $p<n/\gamma$. By applying Lemma~\ref{desigualdad_puntual_Mfrac_M} with $\gamma=\gamma_0$, $p=1$, $q=q_0$ and $s=1$ we obtain
\begin{equation}\label{eq1_teomfracc_pgeneral}
M_{\gamma_0}(f_0w^{-1})(x)\leq M(f_0w^{-q_0})^{1/{q_0}}(x)\left(\int_{\R^n}f_0(y)\,dy\right)^{\gamma_0/n},
\end{equation}
for any non-negative functions $w$ and $f_0\in L^1$. Let us take $w=u^{1/{q_0}}$ and $f_0=f^pu^{1/{q_0}}v$. Then $f_0\in L^1$ and from Lemma~\ref{lema_Mgamma_Mgamma0}  and \eqref{eq1_teomfracc_pgeneral} we have that
\begin{align*}
\frac{M_{\gamma}(fv)(x)}{v(x)}&\leq[v]_{A_1}^{1/{p'}}\left(\frac{M_{\gamma_0}(f^pv)(x)}{v(x)}\right)^{1/p}\\
&= [v]_{A_1}^{1/{p'}}\left(\frac{M_{\gamma_0}(f_0 w^{-1})(x)}{v(x)}\right)^{1/p}\\
&\leq[v]_{A_1}^{1/{p'}}\left[\left(\frac{M(f_0w^{-q_0})}{v^{q_0}}\right)^{1/{q_0}}\left(\int_{\R^n}f_0\right)^{\gamma_0/n}\right]^{1/p}.
\end{align*}
 Then
\begin{align*}
uv^{q_0}\left(\left\{x: \frac{M_\gamma(fv)(x)}{v(x)}>t\right\}\right)^{1/{q_0}}&\leq uv^{q_0}\left(\left\{x: \frac{M(f_0w^{-q_0}v^{-q_0}v^{q_0})(x)}{v^{q_0}(x)}>\left(\frac{[v]_{A_1}^{1-p}t^p}{\left(\int_{\R^n}f_0\right)^{\gamma_0/n}}\right)^{q_0}\right\}\right)^{1/{q_0}}\\
&\leq \frac{[v]_{A_1}^{p-1}}{t^p}\left(\int_{\R^n}f_0\right)^{\gamma_0/n}\left(\int_{\R^n}f_0w^{-q_0}v^{-q_0}uv^{q_0}\right)^{1/q_0},
\end{align*}
from Theorem~\ref{mixed_para_M}. Since $f_0w^{-q_0}v^{-q_0}uv^{q_0}=f_0$ and $\gamma_0/n+1/{q_0}=1$, the last term is equal to $Ct^{-p}\int_{\R^n}f^pu^{1/{q_0}}v\,dx$, as desired.



Finally, we will consider the case $uv^{-q/{p'}}\in A_1$ and $v^q\in A_\infty(uv^{-q/{p'}})$.

Define
 $f_0=fvw$, where $w=u^{1/q}v^{-1/{p'}}$. So that $f_0w^{-1}=fv$ and  $f_0^p=f^pu^{p/q}v$. Then, by applying \eqref{eq: punctual} we get
\begin{align*}
uv^{q/p}\left(\left\{x: \frac{M_\gamma(fv)(x)}{v(x)}>t\right\}\right)^{p/q}&\leq uv^{q/p}\left(\left\{x : \frac{M(f_0^pw^{-q})(x)}{v^q(x)}>\left(\frac{t}{\left(\int_{\R^n}f_0^p\,dy\right)^{\gamma/n}}\right)^q\right\}\right)^{p/q}\\
& \leq uv^{-q/{p'}}v^q\left(\left\{x : \frac{M(f_0^pw^{-q}v^{-q}v^q)(x)}{v^q(x)}>\frac{t^q}{\left(\int_{\R^n}f_0^p\,dy\right)^{q\gamma/n}}\right\}\right)^{p/q}\\
& \leq \frac{C}{t^p}\left(\int_{\R^n}f_0^p\,dy\right)^{p\gamma/n}\left(\int_{\R^n}f_0^pw^{-q}v^{-q}uv^{q/p}\,dy\right)^{p/q},
\end{align*}
from  Theorem~\ref{mixed_para_M}. Notice now that $w^{-q}v^{-q}uv^{q/p}=u^{-1}v^{q/{p'}}v^{-q/{p'}}u=1$, so that the last term is equal to
\[\frac{C}{t^p}\left(\int_{\R^n}f_0^p(y)\,dy\right)^{p(1/q+\gamma/n)}=\frac{C}{t^p}\int_{\R^n}f_0^p(y)\,dy,\]
and the result is proved.
\end{proof}

\section{Mixed inequalities for $I_\gamma$}\label{section:main-fractional integral}

We devote this section to prove Theorem~\ref{teo_main_Tgamma}. In order to do so, we shall use the following result.

\begin{teo}[\cite{CU-M-P}, Thm. 1.7]\label{teo_extrapolacion_carlos}
Given a family $\mathcal{F}$ of pairs of functions that satisfies that there exists a number $p_0$, $0<p_0<\infty$ such that for all $w\in A_\infty$
\[\int_{\R^n}f(x)^{p_0}w(x)\,dx\leq C\int_{\R^n}g(x)^{p_0}w(x)\,dx, \quad\quad \]
for all $(f,g)\in\mathcal{F}$ such the left hand side is finite, and with $C$ depending only on $[w]_{A_\infty}$. Then, for all weights $u,v$ such that $u\in A_1$ and $v\in A_\infty$ we have that
\[\norm{fv^{-1}}_{L^{1,\infty}(uv)}\leq C\norm{gv^{-1}}_{L^{1,\infty}(uv)} \quad (f,g)\in\mathcal{F}.\]
\end{teo}

With this result, we can prove another estimate that will be useful in the proof of Theorem~\ref{teo_main_Tgamma}.

\begin{teo}\label{teo_norma_lqinfty_para_Tgamma}
Let $0<\gamma<n$, $1\leq p<n/\gamma$ and $1/q=1/p-\gamma/n$. If $u,v^{q/p}\in A_1$, then there exists a positive constant $C$  such that
\[\norm{I_\gamma(fv)v^{-1}}_{L^{q,\infty}(uv^{q/p})}\leq C\norm{M_\gamma(fv)v^{-1}}_{L^{q,\infty}(uv^{q/p})},\]
for  every  bounded function $f$ with compact support.
\end{teo}

\begin{proof}
Given a function  $f$, we can assume without loss of generality that $f$ is nonnegative. Denote $f_0=I_\gamma(fv)$ and  $g_0=M_\gamma(fv)$. Define $F=f_0^qv^{-q/{p'}}$ and  $G=g_0^qv^{-q/{p'}}$.
 We shall prove that there exists $p_0\in (0,\infty)$, such that for every $w\in A_\infty$ the inequality
\begin{equation}\label{eq1_teo_norma_lqinfty_para_Tgamma}
\int_{\R^n}F^{p_0}(x)w(x)\,dx\leq C\int_{\R^n}G^{p_0}(x)w(x)\,dx.
\end{equation}
holds.
Indeed, let $p_0$ to be chosen later,  and let  $w\in A_{\infty}$. Since $v^{q/p}\in A_1$, then $v^{-q/p}\in \textrm{RH}_{\infty}$. Also, $v^{-q\alpha/p}\in \textrm{RH}_{\infty}$ for every $\alpha>1$ (in fact, it follows for every positive $\alpha$). Then, taking  $p_0>1/(p-1)$ we have that  $p'/p<p_0$ and then $\alpha=p/{p'}p_0$ is greater than 1, so that $v^{-q/{p'}p_0}=(v^{-q/p})^{\alpha}$ belongs to $\textrm{RH}_{\infty}$. Since $w\in A_\infty$, from Lemma~\ref{lema_producto_Ainf_RHinf} we have that $w_0=v^{-q/{p'}p_0}w$ belongs to $A_\infty$. Hence, from Theorem~\ref{teo_muckenhoupt_Tgamma_Mgamma} we have
\begin{align*}
\int_{\R^n}F^{p_0}(x)w(x)\,dx&=\int_{\R^n}f_0^{qp_0}(x)w_0(x)\,dx\\
&=\int_{\R^n}\left(I_\gamma(fv)(x)\right)^{qp_0}w_0(x)\,dx\\
&\leq C\int_{\R^n}\left(M_\gamma(fv)(x)\right)^{qp_0}w_0(x)\,dx\\
&\leq C\int_{\R^n}g_0^{qp_0}(x)v^{-q/{p'}p_0}(x)w(x)\,dx\\
&=C\int_{\R^n}G^{p_0}(x)w(x)\,dx.
\end{align*}
From Theorem~\ref{teo_extrapolacion_carlos} we have that there exists $C>0$ such that
\[\norm{Fv^{-q/p}}_{L^{1,\infty}(uv^{q/p})}\leq C\norm{Gv^{-q/p}}_{L^{1,\infty}(uv^{q/p})}.\]
Hence
\begin{align*}
\norm{I_\gamma(fv)v^{-1}}_{L^{q,\infty}(uv^{q/p})}&=\norm{(I_\gamma(fv)^qv^{-q})^{1/q}}_{L^{q,\infty}(uv^{q/p})}\\
&=\norm{I_\gamma(fv)^qv^{-q}}_{L^{1,\infty}(uv^{q/p})}^{1/q}\\
&=\norm{Fv^{-q/p}}_{L^{1,\infty}(uv^{q/p})}^{1/q}\\
&\leq C\norm{Gv^{-q/p}}_{L^{1,\infty}(uv^{q/p})}^{1/q}\\
&=C\norm{M_\gamma(fv)^qv^{-q}}_{L^{1,\infty}(uv^{q/p})}^{1/q}\\
&=C\norm{M_\gamma(fv)v^{-1}}_{L^{q,\infty}(uv^{q/p})}.
\end{align*}
\end{proof}



\begin{proof}[Proof of Theorem~\ref{teo_main_Tgamma}]
Let us consider first the case $u,v^{q/p}\in A_1$. If $p=1$, then $q=n/(n-\gamma)$. From Theorem~\ref{teo_muckenhoupt_Tgamma_Mgamma} we have that
\[\int_{\R^n} I_\gamma(fv)^{p_0}(x)w(x)\,dx\leq C\int_{\R^n}M_\gamma(fv)^{p_0}(x)w(x)\,dx,\]
for every $0<p_0<\infty$ and every  $w\in A_\infty$. Then, from Theorem~\ref{teo_extrapolacion_carlos} we get
\[\norm{I_\gamma(fv)^qv^{-q}}_{L^{1,\infty}(uv^q)}\leq C\norm{M_\gamma(fv)^qv^{-q}}_{L^{1,\infty}(uv^q)}.\]
Hence, given $t>0$ we have
\begin{align*}
tuv^q\left(\left\{x\in \R^n: \frac{|I_\gamma(fv)(x)|}{v(x)}>t\right\}\right)^{1/q}&\leq \sup_{t>0}tuv^q\left(\left\{x\in \R^n: \frac{|I_\gamma(fv)(x)|}{v(x)}>t\right\}\right)^{1/q}\\
&=\norm{I_\gamma(fv)v^{-1}}_{L^{q,\infty}(uv^q)}\\
&=\norm{I_\gamma(fv)^qv^{-q}}_{L^{1,\infty}(uv^q)}^{1/q}\\
&\leq C\norm{M_\gamma(fv)^qv^{-q}}_{L^{1,\infty}(uv^q)}^{1/q}\\
&=C\sup_{t>0} tuv^q\left(\left\{x\in \R^n: \frac{M_\gamma(fv)(x)}{v(x)}>t\right\}\right)^{1/q}\\
&\leq C \sup_{t>0} t\frac{C}{t}\int_{\R^n}|f(x)|u^{1/q}(x)v(x)\,dx\\
&=C\int_{\R^n}|f(x)|u^{1/q}(x)v(x)\,dx,
\end{align*}
from Theorem~\ref{teo_main_Mgamma}.

We shall now consider the case $1<p<n/\gamma$. Let $t>0$ be fixed. From Theorems~\ref{teo_main_Mgamma} and ~\ref{teo_norma_lqinfty_para_Tgamma} we get



\begin{align*}
t^puv^{q/p}\left(\left\{x\in \R^n: \frac{|I_\gamma(fv)(x)|}{v(x)}>t\right\}\right)^{p/q}&\leq\left[\sup_{t>0}tuv^{q/p}\left(\left\{x\in \R^n: \frac{|I_\gamma(fv)(x)|}{v(x)}>t\right\}\right)^{1/q}\right]^p\\
&=\norm{I_\gamma(fv)v^{-1}}_{L^{q,\infty}(uv^{q/p})}^p\\
&\leq C\norm{M_\gamma(fv)v^{-1}}_{L^{q,\infty}(uv^{q/p})}^p\\
&=C\left[\sup_{t>0}tuv^{q/p}\left(\left\{x\in \R^n: \frac{M_\gamma(fv)(x)}{v(x)}>t\right\}\right)^{1/q}\right]^p\\
&\leq C\left[\sup_{t>0} t\frac{C}{t}\left(\int_{\R^n}|f(x)|^pu(x)^{p/q}v(x)\,dx\right)^{1/p}\right]^p\\
&= C\int_{\R^n}|f(x)|^pu(x)^{p/q}v(x)\,dx,
\end{align*}
and the desired inequality follows.

Finally, we shall prove the result for the case $uv^{-q/{p'}}\in A_1$ and $v^q\in A_\infty(uv^{-q/{p'}})$. Let us observe first that from the hypotheses we get




$uv^{q/p}=uv^{-q/{p'}}v^q\in A_\infty$, so that  $v^q=uv^{q/p}(uv^{-q/{p'}})^{-1}\in A_\infty$ from Lemma~\ref{lema_equivalencias} and \ref{lema_producto_Ainf_RHinf}. Hence
\begin{align*}
\norm{I_\gamma(fv)v^{-1}}_{L^{q,\infty}(uv^{q/p})}&=\norm{(I_\gamma(fv))^qv^{-q}}_{L^{1,\infty}(uv^{q/p})}^{1/q}\\
&=\norm{(I_\gamma(fv))^qv^{-q}}_{L^{1,\infty}(uv^{-q/{p'}}v^q)}^{1/q}\\
&\leq C \norm{(M_\gamma(fv))^qv^{-q}}_{L^{1,\infty}(uv^{-q/{p'}}v^q)}^{1/q}\\
&= C\norm{M_\gamma(fv)v^{-1}}_{L^{q,\infty}(uv^{-q/{p'}}v^q)}\\
&= C\norm{M_\gamma(fv)v^{-1}}_{L^{q,\infty}(uv^{q/p})},
\end{align*}
from Theorem~\ref{teo_extrapolacion_carlos}. Then,
\begin{align*}
t^puv^{q/p}\left(\left\{x\in \R^n: \frac{|I_\gamma(fv)(x)|}{v(x)}>t\right\}\right)^{p/q}&\leq\left[\sup_{t>0}tuv^{q/p}\left(\left\{x\in \R^n: \frac{|I_\gamma(fv)(x)|}{v(x)}>t\right\}\right)^{1/q}\right]^p\\
&=\norm{I_\gamma(fv)v^{-1}}_{L^{q,\infty}(uv^{q/p})}^p\\
&\leq C\norm{M_\gamma(fv)v^{-1}}_{L^{q,\infty}(uv^{q/p})}^p\\
&=C\left[\sup_{t>0}tuv^{q/p}\left(\left\{x\in \R^n: \frac{M_\gamma(fv)(x)}{v(x)}>t\right\}\right)^{1/q}\right]^p\\
&\leq C\left[\sup_{t>0} t\frac{C}{t}\left(\int_{\R^n}|f(x)|^pu(x)^{p/q}v(x)\,dx\right)^{1/p}\right]^p\\
&= C\int_{\R^n}|f(x)|^pu(x)^{p/q}v(x)\,dx,
\end{align*}
from which the result follows.
\end{proof}

\section{Applications}\label{section: applications}

In this section we will give some applications that can be obtained from the inequalities we have proved so far. More precisely, we will consider mixed weak inequalities of Sawyer type for commutators of Calder\'on-Zygmund operators and fractional integral operators, both with Lipschitz symbol.

Let us recall that given $0< \delta\leq1$, a function $b\in L^1_\textit{loc}$ belongs to the class $\Lambda(\delta)$ if there exists a positive constant $C$ such that
\[|b(x)-b(y)|\leq C|x-y|^{\delta},\]
for every $x,y\in\R^n$. The smallest constant for which the condition above holds is denoted by $\norm{b}_{\Lambda(\delta)}$.

 Let $T$ be an operator of convolution type, that is, $T$ is bounded on $L^2(\R^n)$ and such that for $f\in L^2$ with compact support,
 \begin{equation}\label{rep_T_fuera_del_soporte}
Tf(x)=\int_{\R^n}K(x-y)f(y)\,dy ,\quad\quad x\notin \supp f,
\end{equation}
with $K$ a measurable function defined away from the origin.
We say that $T$ is a \emph{Calder\'on-Zygmund operator}   if  $K$ is a \emph{standard kernel}, which means that
 $K:\mathbb R^n\times \mathbb R^n\backslash \Delta\to\mathbb C$  satisfies a size condition given by 
\[|K(x,y)|\lesssim \frac{1}{|x-y|^n},\]
and the smoothness conditions, usually called Lipschitz conditions,
\begin{equation}\label{eq:prop del nucleo}
|K(x,y)-K(x,z)|\lesssim \frac{|x-z|}{|x-y|^{n+1}},\quad \textrm{ if } |x-y|>2|y-z|,
\end{equation}
\[|K(x,y)-K(w,z)|\lesssim\frac{|x-w|}{|x-y|^{n+1}},\quad \textrm{ if } |x-y|>2|x-w|.\]
As usual, the notation $f(t)\lesssim g(t)$ means that there exists a constant $C$ such that $f(t)\leq Cg(t)$ for every $t$ in the domain. The constant $C$ may change on each occurrence.\\

We say that a Calder\'on-Zygmund operator $T$ is a \emph{Calder\'on-Zygmund singular integral} operator if
\[Tf(x)=\lim_{\varepsilon\to 0}\int_{|y-x|>\varepsilon}f(y)K(x-y)\,dy=\lim_{\varepsilon\to 0}T_\varepsilon f(x)\]
for almost every $x$.

Given $b\in L^1_\textit{loc}$ and a Calder\'on-Zygmund operator $T$, the commutator operator is defined by
\[[b,T]f(x)=T_bf(x)=b(x)Tf(x)-T(bf)(x),\]
for adequate functions $f$.

One can also consider higher order commutators, which are defined by induction as follows: $T_b^0=T$, and $T_b^m=[b, T_b^{m-1}]$, for $m\geq 1$.

It is not difficult to see, by using \eqref{rep_T_fuera_del_soporte}, that
\[T_b^mf(x)=\int_{\R^n}(b(x)-b(y))^mK(x-y)f(y)\,dy ,\quad\quad x\notin \supp f.\]
However, if $T$ is a Calder\'on-Zygmund operator that satisfies the stronger condition
\begin{equation}\label{rep_T_en_todos_lados}
Tf(x)=\int_{\R^n}K(x-y)f(y)\,dy, \quad\quad \textrm{a.e}\hspace{0.1truecm} x,
\end{equation}
then for every natural number $m$ we have a similar representation for its commutator of order $m$. That is,

\begin{equation}\label{rep_Tbm_en_todos_lados}
T_b^m f(x)=\int_{\R^n}(b(x)-b(y))^mK(x-y)f(y)\,dy, \quad\quad \textrm{a.e}\hspace{0.1truecm} x.
\end{equation}

We shall require condition \eqref{rep_T_en_todos_lados} in order to prove Theorem~\ref{teo_mixed_commutatorCZ_lip_ordenm}.

In \cite{CU-M-P} the authors proved mixed weak estimates for Calder\'on-Zygmund singular integral operators. The precise statement is contained in the following theorem.

\begin{teo}\cite[Thm. 1.3]{CU-M-P}\label{teo1.7}
If $u,v$ are weights such that $u,v\in A_1$, or $u\in A_1$ and $v\in
A_{\infty}(u)$, then there exists a positive constant $C$ such that for every
$t>0$,
\[uv\left(\left\{x\in\R^n: \frac{|T(fv)(x)|}{v(x)}>t\right\}\right)\leq \frac{C}{t}\int_{\R^n}|f(x)|u(x)v(x)\,dx,\]
where
 $T$ is any Calder\'on-Zygmund operator.
\end{teo}

Later on, in \cite{bcp} we have extended this estimate to commutators of such operators with BMO symbol, for the case $u\in A_1$ and $v\in A_\infty(u)$. More precisely, we have proved the following result.

\begin{teo}[\cite{bcp}]\label{teo_BCP_commutator}
Let $u,v$ weights such that $u\in A_1$ and $v\in
A_{\infty}(u)$, $b\in \textrm{BMO}$ and $m\geq 1$. Then, there exists a positive constant $C$ such that for every
$t>0$,
\[uv\left(\left\{x\in\R^n: \frac{|T_b^m(fv)(x)|}{v(x)}>t\right\}\right)\leq C\int_{\R^n}\Phi_m\left(\frac{|f(x)|}{t}\right)u(x)v(x)\,dx,\]
where $T$ is any Calder\'on-Zygmund operator and $\Phi_m(t)=t(1+\log^+t)^m$.
\end{teo}

The next pointwise estimate will allow us to obtain the desired mixed weak estimates for commutators of Calder\'on-Zygmund singular integral operators with Lipschitz symbol.

\begin{lema}\label{lema_puntual_para_Tb}
Let $0<\delta\leq 1$, $b\in \Lambda(\delta)$ and $T$ a Calder\'on-Zygmund singular integral operator. Then, for almost every $x\in \R^n$ we have that
\[|[b,T](f)(x)|\leq C\norm{b}_{\Lambda(\delta)}I_\delta |f|(x),\]
for every bounded function $f$ with compact support.
\end{lema}

\begin{proof}
We will first assume that $b\in L^{\infty}$. Let us write
\[b(x)Tf(x)-T(bf)(x)=\lim_{\varepsilon\to 0}\int_{|y-x|>\varepsilon}b(x)K(x-y)f(y)\,dy-\lim_{\varepsilon\to 0}\int_{|y-x|>\varepsilon}b(y)f(y)K(x-y)\,dy.\]
Then, since $b\in L^{\infty}$ both limits exist, and we can pass to the limit of the difference. So, it will be enough to prove that
\[\left|\int_{|y-x|>\varepsilon}(b(x)-b(y))K(x-y)f(y)\,dy\right|\leq C\norm{b}_{\Lambda(\delta)}I_\delta |f|(x),\]
uniformly on $\varepsilon$.

Indeed, given $\varepsilon>0$
\begin{align*}
\left|\int_{|y-x|>\varepsilon}(b(x)-b(y))K(x-y)f(y)\,dy\right|&\leq \int_{|y-x|>\varepsilon}|b(x)-b(y)||K(x-y)||f(y)|\,dy\\
&\leq C\norm{b}_{\Lambda(\delta)}\int_{|y-x|>\varepsilon}\frac{|f(y)|}{|x-y|^{n-\delta}}\,dy\\
&\leq C\norm{b}_{\Lambda(\delta)}\int_{\R^n}\frac{|f(y)|}{|x-y|^{n-\delta}}\,dy\\
&= C\norm{b}_{\Lambda(\delta)} I_\delta |f|(x),
\end{align*}
where we have used the Lipschitz condition of $b$ and the size condition of $K$.

Finally, if $b$ is not in $L^\infty$ we define, for $N\in \N$, $b_N(x)=b(x)$ if $b(x)\in [-N,N]$, $b_N(x)=-N$ if $b(x)<-N$ and $b_N(x)=N$ if $b(x)>N$. Then, it is not difficult to see that $|b_N(x)-b_N(y)|\leq |b(x)-b(y)|$, which implies that $b_N\in \Lambda(\delta)$ and $\norm{b_N}_{\Lambda(\delta)}\leq \norm{b}_{\Lambda(\delta)}$. So, we obtain the desired estimate for a fixed $N$ with constant independent of $N$. Since $(b_N(x)-b_N(y))K(x-y)f(y)$ is uniformly bounded by a function of $L^1(\R^n)$, for almost every $x$, the Dominated Convergence Theorem gives us the same estimate for $b$. This completes the proof.
\end{proof}

In the same way as for Calder\'on-Zygmund operators, one can define commutators of fractional integral operators. For a nonnegative integer $m$ and $0<\gamma<n$ we write $I_{\gamma,b}^0=I_\gamma$ and $I_{\gamma,b}^m=[b,I_{\gamma,b}^{m-1}]$. Because of the definition of the operator $I_\gamma$, we have the corresponding estimate \eqref{rep_Tbm_en_todos_lados} for this case, with $K(x-y)=|x-y|^{\gamma-n}$. The following lemma gives a pointwise estimate between fractional higher order commutators and the fractional integral operator.

\begin{lema}
Let $0<\gamma<n$, $m\in \N$, $0<\delta<\min\{(n-\gamma)/m,1\}$ and $b\in \Lambda(\delta)$. Then, for almost every $x\in \R^n$ we have that
\[|I_{\gamma,b}^m(f)(x)|\leq \norm{b}_{\Lambda(\delta)}^mI_{\gamma+m\delta} |f|(x),\]
for every bounded function $f$ with compact support.
\end{lema}

\begin{proof}
As before, let us first assume that $b\in L^{\infty}$. Then,
\begin{align*}
|I_{\gamma,b}^mf(x)|&=\left|\int_{\R^n}(b(x)-b(y))^{m}\frac{f(y)}{|x-y|^{n-\gamma}}\,dy\right|\\
&\leq \norm{b}_{\Lambda(\delta)}^m\int_{\R^n}\frac{|f(y)|}{|x-y|^{n-(\gamma+m\delta)}}\,dy\\
&\leq \norm{b}_{\Lambda(\delta)}^m I_{\gamma+m\delta}|f|(x).
\end{align*}
If $b$ is not in $L^\infty$ we can proceed as in the previous lemma.
\end{proof}

With these two estimates we can easily obtain the two following results.

\begin{teo}
Let $0<\delta\leq 1$, $1\leq p< n/\delta$ and q such that $1/q=1/p-\delta/n$. Let $u,v$ be weights satisfying $u, v^{q/p}\in A_1$ or $uv^{-q/{p'}}\in A_1$ and $v^q\in A_\infty(uv^{-q/{p'}})$. If $b\in \Lambda(\delta)$ and $T$ is a Calder\'on-Zygmund singular integral operator, then the estimate
\[uv^{q/p}\left(\left\{x\in \R^n: \frac{|T_b(fv)(x)|}{v(x)}>t\right\}\right)^{1/q}\leq \frac{C}{t}\norm{b}_{\Lambda(\delta)}\left(\int_{\R^n}|f(x)|^pu^{p/q}(x)v(x)\,dx\right)^{1/p}\]
holds for some constant $C$ and every $t>0$, for all bounded function $f$ with compact support.
\end{teo}

\begin{obs} Let us notice that the estimate obtained for $p=1$ in this theorem is different from the one in Theorem~\ref{teo_BCP_commutator}. This is because of the nature of the symbol: when we consider Lipschitz symbol we get a control by the $L^1$-norm of $f$ with respect to the measure $d\mu = u^{1/q}(x)v(x)\,dx$, which is not as big as the $L^{\Phi}$-norm of $f$ with the corresponding measure when $q=1$ as well. So, since the Lipschitz condition is ``smoother'' than the BMO condition, it allows us obtain a better bound.
\end{obs}

\begin{teo}
Let $0<\gamma<n$, $m\in \N$, $0<\delta<\min\{ (n-\gamma)/m,1\}$ and $b\in\Lambda(\delta)$. Let $1\leq p< n/(\gamma+m\delta)$ and q such that $1/q=1/p-(\gamma+m\delta)/n$. Let $u,v$ be weights satisfying $u, v^{q/p}\in A_1$ or $uv^{-q/{p'}}\in A_1$ and $v^q\in A_\infty(uv^{-q/{p'}})$. Then the estimate
\[uv^{q/p}\left(\left\{x\in \R^n: \frac{|I_{\gamma,b}^m(fv)(x)|}{v(x)}>t\right\}\right)^{1/q}\leq \frac{C}{t}\norm{b}_{\Lambda(\delta)}^m\left(\int_{\R^n}|f(x)|^pu^{p/q}(x)v(x)\,dx\right)^{1/p}\]
holds for some constant $C$ and every $t>0$, for all bounded function $f$ with compact support.
\end{teo}

\begin{obs}
If we have a Calder\'on-Zygmund operator $T$ which satisfies \eqref{rep_T_en_todos_lados}, then the estimate given in Lemma~\ref{lema_puntual_para_Tb} can be easily generalized for higher order commutators, since they will satisfy \eqref{rep_Tbm_en_todos_lados}. In this case, the estimate that we obtain is
\[|T_b^mf(x)|\leq C\norm{b}_{\Lambda(\delta)}^mI_{m\delta }|f|(x),\]
for almost every $x$. This leads us to the following result.
\end{obs}

\begin{teo}\label{teo_mixed_commutatorCZ_lip_ordenm}
Let $T$ a Calder\'on-Zygmund operator that satisfies \eqref{rep_T_en_todos_lados}. Let $0<\delta<\min\{n/m,1\}$, $b\in \Lambda(\delta)$, $1\leq p<n/(m\delta)$ and $q$ such that $1/q=1/p-m\delta/n$. Then, if $u,v$ are weights satisfying $u,v^{q/p}\in A_1$ or $uv^{-q/{p'}}\in A_1$ and $uv^q\in A_\infty(uv^{-q/{p'}})$ we have that there exists a positive constant $C$ for which the inequality
 \[uv^{q/p}\left(\left\{x\in \R^n: \frac{|T_b^m(fv)(x)|}{v(x)}>t\right\}\right)^{1/q}\leq \frac{C}{t}\norm{b}_{\Lambda(\delta)}^m\left(\int_{\R^n}|f(x)|^pu^{p/q}(x)v(x)\,dx\right)^{1/p}\]
 holds for every $t>0$ and every bounded function with compact support $f$.
\end{teo}


\end{document}